\documentclass{amsart}
\usepackage{graphicx}
\usepackage[mathscr]{eucal}
\usepackage{amssymb}
\newtheorem{thm}{Theorem}[section]

\newtheorem{lem}[thm]{Lemma}
\newtheorem{prop}[thm]{Proposition}
\theoremstyle{definition}
\newtheorem{defn}[thm]{Definition}
\theoremstyle{remark}
\newtheorem{rem}[thm]{Remark}

\newcommand{\set}[1]{\left\{#1\right\}}

\newcommand{\To}{\longrightarrow}

\newcommand{\Li}{\mathcal{L}}
\newcommand{\ML}{\mathcal{M}\mathcal{L}}
\newcommand{\F}{\mathcal{F}}
\newcommand{\pr}{{}^{\prime}}
\newcommand{\ts}{\tau_s}
\newcommand{\tw}{\tau_w}
\newcommand{\Lp}{\Li\pr}
\newcommand{\Cs}{C_{\sim}^b}
\newcommand{\CR}{\mathcal{C}\mathcal{R}}
\newcommand{\emset}{\emptyset}
\newcommand{\Pri}{\textrm{Prim}}
\begin{document}
\baselineskip=18pt
\title{Quotient spaces determined by algebras of continuous functions}
\author{Aldo J. Lazar}
\address{School of Mathematical Sciences\\
         Tel Aviv University\\
         Tel Aviv 69778, Israel}
\email{aldo@post.tau.ac.il}

\thanks{}%

\date{}%
\begin{abstract}

   We prove that if $X$ is a locally compact $\sigma$-compact space then on its quotient, $\gamma(X)$ say, determined by the algebra
   of all real valued bounded continuous functions on $X$, the quotient topology and the completely regular topology defined
   by this algebra are equal. It follows from this that if $X$ is second countable locally compact then $\gamma(X)$
   is second countable locally compact Hausdorff if and only if it is first countable. The interest in these results
   originated in \cite{A} and \cite{EW98} where the primitive ideal space of a $C^*$-algebra was considered.

\end{abstract}
\maketitle
\section{Introduction} \label{S:Intro}

The primitive ideal space, $\Pri (A)$, of a $C^*$-algebra $A$ with its hull-kernel topology has some pleasant
properties: it is a locally compact Baire space, see \cite[Corollary 3.3.8 and Corollary 3.4.13]{Di}. However, no
better separation property than $T_0$ can be expected in general. The absence of the Hausdorff separation property
in $\Pri (A)$ justified a study in \cite{L} of the collection of all the closed limit sets of a topological space.
It should be added that the closed limit subsets of $\Pri (A)$ correspond, by one of the bijections detailed in
\cite[Proposition 3.2.1]{Di}, to some distinguished ideals of $A$. An investigation of the topologies on the class
of these ideals that mirror topologies on the collection of closed limit sets of $\Pri (A)$ was performed in
\cite{A}. We intend here to pursue the study begun in \cite{AS} of the topologies on the quotient of $\Pri (A)$
determined by the algebra of the bounded scalar functions. We adopted a purely topological setting so no knowledge
of the theory of $C^*$-algebras is needed for reading this paper. The method we chose is to substitute for the
quotient of a possibly non-Hausdorff space $X$, via a homeomorphism, a quotient of a certain Hausdorff hyperspace of
$X$.

In the following $X$ denotes a locally compact space that is, $X$ is a space in which every point has a
neighbourhood base of compact sets. The algebra of all complex valued bounded continuous functions on $X$, denoted
$C^b(X)$, induces an equivalence relation on $X$: $x_1\sim x_2$ if $f(x_1) = f(x_2)$ for every $f\in C^b(X)$. We let
$\gamma(X)$ be the quotient space of this equivalence relation; the quotient map $q: X\To \gamma(X)$ was called in
\cite[section III.3]{DH} the complete regularization of $X$ and this construction was discussed in \cite{AS} for the
special case of the primitive ideal space of a $C^*$-algebra. As in \cite{AS}, we endow $\gamma(X)$ with the
topology $\tau_{cr}$, the weak topology defined by the bounded continuous functions on $X$ viewed as functions on
the quotient. Another natural topology on $\gamma(X)$ is the quotient topology $\tau_q$. Always $\tau_{cr}\subset
\tau_q$; it was shown in \cite{AS} that if $X$ is compact or if $q$ is open for $\tau_q$ or $\tau_{cr}$ then $\tau_q
= \tau_{cr}$. The question whether these topologies are equal for any locally compact space was left open there. Of
course, $C^b(\gamma(X))$ is always the same for both topologies. We shall prove in Section \ref{S:equi} that if $X$
is $\sigma$-compact then $\tau_{cr} = \tau_q$. Recently D. W. B. Somerset found an example of a locally compact
space $X$ for which these two topologies are different. The example appears in an appendix to this paper and only
its last paragraph, where it is shown that the topological space constructed there is homeomorphic to the primitive
ideal space of a $C^*$-algebra, needs a minimum knowledge of operator algebra theory.

In Section \ref{S:CR} we discuss the class of second countable locally compact spaces $X$ for which $\gamma(X)$ is a
second countable locally compact Hausdorff space. Such spaces when serving as primitive ideal spaces of
$C^*$-algebras were considered in \cite{EW98} and \cite{EW01}. We give a characterization of these spaces by using
the tools developed in Section \ref{S:equi}.

The family of all closed subsets of $X$ will be denoted by $\F(X)$ and its subfamily that consists of all the
nonempty closed subsets of $X$ will be denoted $\F \pr(X)$. We shall equip $\F(X)$ with two topologies: the Fell
topology, denoted here $\ts$, that was defined in \cite{F}, and the lower semifinite topology of Michael, which we
denote $\tw$, see \cite{M}. A base for $\ts$ consists of the family of all the sets
\[
 \mathcal{U}(C,\Phi) := \{S\in \F(X) \mid S\cap C = \emset, S\cap O\neq \emset, O\in \Phi\}
\]
where $C$ is a compact subset of $X$ and $\Phi$ is a finite family of open subsets of $X$. The hyperspace $(\F(X),
\ts)$ is always compact Hausdorff (\cite[Lemma 1 and Theorem 1]{F}). If $X$ is second countable then this hyperspace
is metrizable (\cite[Lemme 2]{D}). The family of all the sets $\mathcal{U}(\emset,\Phi)$ is a base for $\tw$. The
map $\eta_X$ given by $\eta_X(x) := \overline{\{x\}}$ from $X$ to $\F(X)$ is $\tw$-continuous; it is not
$\ts$-continuous in general. It is one-to-one exactly when $X$ is a $T_0$ space.

A subset $S$ of $X$ is called a limit subset if there is a net in $X$ that converges to all the points of $S$. The
family of all the closed limit subsets of $X$ will be denoted $\Li (X)$; it is a compact Hausdorff space with its
relative $\ts$-topology, metrizable if $X$ is second countable, see \cite[Th\'{e}or\`{e}me 12 and Lemme 2]{D}. We
put $\Lp(X) := \Li(X)\setminus \{\emptyset\}$. Then $(\Lp(X),\ts)$ is a locally compact Hausdorff space; if $X$ is
compact then one easily sees that $\emset$ is an isolated point of $(\Li(X),\ts)$ hence $(\Lp(X),\ts)$ is compact.
The family of all (closed) maximal limit subsets is denoted $\ML(X)$ and $\ML^s(X)$ stands for its $\ts$-closure in
$\Lp(X)$.

\section{equivalence relations on $\Lp(X)$}\label{S:equi}

We define on $\Li\pr(X)$ an equivalence relation: we say that $A\underset{1}\sim B$ if there is a finite sequence
$\{F_i \mid 0\leq i\leq n\}$ in $\Lp(X)$ with $F_0 = A$ and $F_n = B$ such that $F_i\cap F_{i+1} \neq \emptyset$, \
$0\leq i\leq n-1$. Let now $\Cs(\Lp(X))$ be the algebra of all $\mathbb{C}$-valued bounded $\ts$-continuous
functions on $\Lp(X)$ that are constant on the equivalence classes with respect to $\underset{1}\sim$. For $A,B\in
\Lp(X)$ we shall say that $A\underset{2}\sim B$ if $f(A) = f(B)$ for every $f\in \Cs(\Lp(X))$. Then
$\underset{2}\sim$ is an equivalence relation on $\Lp(X)$ and by definition $A\underset{1}\sim B$ implies
$A\underset{2}\sim B$. On $Q(X) := \Lp(X)/\underset{2}\sim$ we shall consider two topologies: the quotient topology,
$\tau_Q$, defined by the quotient map $Q : \Lp(X)\To Q(X)$ when $\Lp(X)$ is endowed with the $\ts$ topology and the
completely regular topology $\tau_{CR}$ given by the functions of $\Cs(\Lp(X))$ considered as functions on $Q(X)$.
Obviously $\tau_{CR}\subset \tau_Q$; it will follow from subsequent results that the question of equality between
these two topologies parallels the situation between $(\gamma(X),\tau_{cr})$ and $(\gamma(X),\tau_q)$.

Let $f\in C^b(X)$; then $f$ is constant on every closed limit subset of $X$. Define $f^{\Li}$ on $\Li\pr(X)$ by
$f^{\Li}(S) = f(x)$ where $x$ is any point of $S\in \Li\pr(X)$. Then $f^{\Li}$ is $\tw$-continuous of $\Li\pr(X)$,
thus $f^{\Li}\in \Cs(\Lp(X))$. Indeed, if $D$ is an open subset of $\mathbb{C}$ then $U := \{x\in X \mid f(x)\in
D\}$ is open hence $\{S\in \Li\pr(X) \mid f^{\Li}(S)\in D\} = \{S\in~ \Li\pr(X) \mid S\cap U \neq \emptyset \}$ is
in $\tw$.

We have a converse to the statement about the continuity of $f^{\Li}$ but first we need a lemma about
$\ts$-convergence in $\Lp(X)$. It is included in \cite[Lemma H.2]{W} but we give below its simple proof for the sake
of self sufficiency.

\begin{lem} \label{L:convergence}

   Let $\{S_{\alpha} \mid \alpha\in \mathcal{A}\}$ be a net in $\Lp(X)$ that $\tau_s$-converges to $S\in \Li(X)$. If
   $x_{\alpha}\in S_{\alpha}$ for $\alpha\in \mathcal{A}$ and $\{x_{\alpha}\}$ converges to $x\in X$ then $x\in S$.

\end{lem}

\begin{proof}

   Assuming that $x\notin S$ we let $K$ be a compact neighbourhood of $x$ disjoint from $S$. Thus $S\in \{T\in
   \Li(X) \mid T\cap K = \emptyset\}$, hence eventually $S_{\alpha}\cap K = \emptyset$ and $\set{x_{\alpha}}$
   cannot converge to $x$, a contradiction.

\end{proof}

\begin{thm} \label{T:iso}

   The map $f\To f^{\Li}$ is an isomorphism of $C^{b}(X)$ onto $\Cs(\Lp(X))$.

\end{thm}

\begin{proof}

   We have to prove only the surjectivity of the map. Let $g\in \Cs(\Lp(X))$ and define $f$ on $X$ by $f(x) =
   g(S)$, $S$ being any element in $\Lp(X)$ such that $x\in S$, for instance the closure of $\set{x}$. Then $f$ is
   well defined and we are going to show that it is continuous. Once this will be done we clearly shall have $f^{\Li} = g$
   and the proof will be finished. Let $D$ be an open subset of $\mathbb{C}$ and $x\in f^{-1}(D)$. We
   claim that there is a neighbourhood of $x$ contained in $f^{-1}(D)$. If not then there is a net $\set{x_{\alpha}
   \mid \alpha\in \mathcal{A}}$ that converges to $x$ but $f(x_{\alpha})\notin D$ for every $\alpha\in \mathcal{A}$.
   We choose $S_{\alpha}\in \Lp(X)$ such that $x_{\alpha}\in S_{\alpha}$ so that $f(x_{\alpha}) =
   g(S_{\alpha})$ for each $\alpha\in \mathcal{A}$. The net $\set{S_{\alpha}}$ has a subnet $\set{S_{\alpha\pr}}$ that $\ts$-converges to some $S\in
   \Li(X)$. By Lemma \ref{L:convergence} $x\in S$ hence $S\neq \emptyset$ and $g(S) = f(x)\in D$. The continuity of $g$
   implies that eventually
   $f(x_{\alpha\pr}) = g(S_{\alpha\pr})\in D$ contradicting our choice of the net $\set{x_{\alpha}}$.

\end{proof}

\begin{rem}

   It follows from the $\tw$-continuity of $\eta_X$, the definition of $f^{\Li}$ for $f\in C^b(X)$ and the preceding
   proof that in the definition of $\Cs(\Lp(X))$ we can substitute $\tw$-continuity for $\ts$-continuity.

\end{rem}

We can now define a one-to-one map $\chi$ from $\gamma(X)$ onto $Q(X)$ as follows: for $x\in X$ we let $\chi(q(x))
:= Q(S)$ where $S$ is any closed limit set that contains $x$. It readily follows from Theorem \ref{T:iso} that the
map is well defined and it has the stated properties. As a direct consequence of the definitions we have for every
$g\in C^b(\gamma(X))$ that $(g\circ q)^{\Li} = g\circ \chi^{-1}\circ Q$. It is clear that $\chi$ is a homeomorphism
of $(\gamma(X),\tau_{cr})$ onto $(Q(X),\tau_{CR})$.

\begin{prop}\label{P:quot}

   The map $\chi$ defined above is a homeomorphism from $(\gamma(X),\tau_q)$ onto $(Q(X),\tau_Q)$.

\end{prop}

\begin{proof}
Let $O\subset Q(X)$ be open in the quotient topology. Then $U := q^{-1}(\chi^{-1}(O))$ is the union of all the
elements of $Q^{-1}(O)$ and we claim that $U$ is open. Otherwise, there are $x\in U$ and a net $\set{x_{\alpha}}$ in
$X\setminus U$ that converges to $x$. For each index $\alpha$ we choose $S_{\alpha}$ such that $x_{\alpha}\in
S_{\alpha}$. This compels each $S_{\alpha}$ to belong to the closed subset $\Lp(X)\setminus Q^{-1}(O)$ of $\Lp(X)$.
By passing to a subnet if necessary we may suppose that $\set{S_{\alpha}}$ $\ts$-converges to some $S\in \Li(X)$.
Lemma \ref{L:convergence} yields $x\in S$ thus $S\in Q^{-1}(O)$. On the other hand, $S\notin Q^{-1}(O)$ as the
$\ts$-limit of $S_{\alpha}$, a contradiction. Therefore $U$ is an open subset of $X$, $\chi^{-1}(O)$ is a
$\tau_q$-open subset of $\gamma(X)$ and the continuity of $\chi$ is established.

Suppose now that $V\subset \gamma(X)$ is $\tau_q$-open. Then

\[
 Q^{-1}(\chi(V)) = \set{S\in \Lp(X) \mid S\subset q^{-1}(V)} = \set{S\in \Lp(X) \mid S\cap q^{-1}(V)\neq \emptyset}
\]

is $\tau_w$-open hence $\tau_s$-open. Thus $\chi(V)$ is $\tau_Q$-open.
\end{proof}

Before stating the main result of this section we need a lemma.

\begin{lem}\label{L:sigma}

   If $X$ is $\sigma$-compact then $(\Lp(X),\tau_s)$ is $\sigma$-compact too.

\end{lem}

\begin{proof}

   Let $X = \cup_{n=1}^{\infty} K_n$ where each $K_n$ is compact and set $L_n = \{S\in \Lp(X) \mid S\cap K_n \neq
   \emset \}$. Clearly $\Lp(X) = \cup_{n=1}^{\infty} L_n$ and we are going to show that each $L_n$ is $\tau_s$-compact. Let
   $\set{S_{\alpha}}$ be a net in $L_n$ and without loss of generality we shall suppose that it $\tau_s$-converges to some
   $S\in \Li(X)$. For each $\alpha$ choose $x_{\alpha}\in S_{\alpha}\cap K_n$. By passing to a subnet we may suppose
   that $\set{x_{\alpha}}$ converges to some $x\in K_n$. Lemma \ref{L:convergence} yields $x\in S$ thus $S\in L_n$.

\end{proof}

\begin{thm} \label{T:equal}

  If $X$ is $\sigma$-compact then $\tau_{cr} = \tau_q$ and $\gamma(X)$ is paracompact.

\end{thm}

\begin{proof}

   In view of Proposition \ref{P:quot} and the remarks preceding its statement it will suffice to show that
   $\tau_{CR} = \tau_Q$ and that $Q(X)$ is paracompact. Now, $\Lp(X)$ with the $\tau_s$ topology is locally compact
   Hausdorff and $\sigma$-compact by
   Lemma \ref{L:sigma} hence Lindel\"{o}f. Its quotient space $(Q(X),\tau_Q)$ is Hausdorff since the real
   valued bounded continuous functions on $Q(X)$ separate its points. It follows from Theorem 1 of \cite{Mo} that
   $Q(X)$ with its $\tau_Q$ topology is a paracompact space. In particular it is also completely regular, hence
   $\tau_{CR} = \tau_Q$ and we are done.

\end{proof}

\begin{rem}

   In all of the above we could have used the space $\ML^s(X)$ instead of $\Lp(X)$.

\end{rem}

We are going to treat another situation when the two topologies on $\gamma(X)$ coincide but first we have to
introduce a new relation on the space $X$ that was considered in \cite{AS} for the primitive ideal space of a
$C^*$-algebra. For $x,y\in X$ we shall write $x\underset{H}\sim y$ if $x$ and $y$ cannot be separated by disjoint
open subsets of $X$. This is the same as saying that there is a closed limit subset of $X$ to which both $x$ and $y$
belong. In general this is not a transitive relation. Clearly, if $\underset{H}\sim$ is an equivalence relation then
each equivalence class for it is the union of all the elements in an equivalence class with respect to
$\underset{1}\sim$ on $\Lp(X)$ and each such union is an equivalence class for $\underset{H}\sim$. The following
result is the same as Proposition 3.2 of \cite{AS} when $X$ is the primitive ideal space of a $C^*$-algebra but the
proof below differs in part from that given there.

\begin{prop} \label{P:open}

   Suppose $\underset{H}\sim$ is an open equivalence relation. Then each equivalence class is a maximal limit set and
   each maximal limit set of $X$ is an equivalence class for $\underset{H}\sim$. The relations $\sim$ and
   $\underset{H}\sim$ are the same, $\tau_q = \tau_{cr}$, the quotient map $q$ is open and
   $\gamma(X)$ is a locally compact Hausdorff space.

\end{prop}

\begin{proof}

   Let $S$ be an equivalence class for $\underset{H}\sim$. We are going to show that for every finite set $\set{x_i
   \mid 1\leq i\leq n}\subset S$ and every neighbourhood $V_i$ of $x_i$, $1\leq i\leq n$, we have $\cap_{i=1}^n
   V_i\neq \emset$; by \cite[Lemme 9]{D} this will imply that $S$ is a limit set. The claim is obviously valid for
   any pair of points of $S$. Suppose that it is true for any subset of $n-1$ points of $S$ and let $x_i\in S$ with
   an arbitrary neighbourhood $V_i$, $1\leq i\leq n$. Let $U$ be the open saturation of $V_n$ for
   $\underset{H}\sim$. Then $x_i\in U$ and $U_i := V_i\cap U$ is a neighbourhood of $x_i$, $1\leq i\leq n-1$. By the
   induction hypothesis
   $W := \cap_{i=1}^{n-1}U_i$ is a non-void open set. Choose $x\in W$; there is $y\in V_n$ such that $x\underset{H}\sim y$.
   By the definition of $\underset{H}\sim$ we have $\cap_{i=1}^nV_i\supset W\cap V_n\neq \emset$ and the claim is
   established. Now if $S\pr$ is a limit set with $S\pr\supset S$ then each point of $S\pr$ is
   $\underset{H}\sim$-equivalent to each point of $S$ hence $S\pr = S$. Thus $S$ is a maximal limit set. Since all the points of a
   maximal limit set are $\underset{H}\sim$-equivalent no two different maximal limit sets can intersect and each
   maximal limit set is an equivalence class for $\underset{H}\sim$.

   From here on we follow the proof of \cite[Proposition 3.2]{AS}. If $S_1$ and $S_2$ are two different
   $\underset{H}\sim$-classes and $x_i\in S_i$ then $x_1$ and $x_2$ have two disjoint open neighbourhoods $V_1$ and
   $V_2$, respectively. No point of $V_1$ can be $\underset{H}\sim$-equivalent to any point of $V_2$ because
   $V_1\cap V_2 = \emset$. Hence the open quotient map of $X$ onto $X/\underset{H}\sim$ maps $V_1$ and $V_2$ onto
   two disjoint neighbourhoods of $S_1$ and $S_2$ respectively, which means that the quotient space is Hausdorff. Since the
   quotient map is open the quotient space is also locally compact. Clearly if $x\underset{H}\sim y$ then $x\sim y$. By the
   complete regularity of $X/\underset{H}\sim$ we conclude that if $x$ and $y$ are not
   $\underset{H}\sim$-equivalent they are also not $\sim$-equivalent. Thus $\underset{H}\sim$ and $\sim$ are
   identical. Since $(\gamma(X), \tau_q)$ and $(\gamma(X), \tau_{cr})$ have the same bounded continuous functions
   and both are completely regular, the identity map is a homeomorphism.

\end{proof}

\section{$\CR$-spaces} \label{S:CR}

In this section we shall discuss a class of locally compact second countable spaces. A locally compact second
countable space has a countable base consisting of interiors of compact subsets since the family of the interiors of
all the compact subsets is a base and as such it must contain a countable base by \cite[Problem 1.F]{K}. Thus such a
space $X$ is $\sigma$-compact and by Theorem \ref{T:equal} we have only one topology on $\gamma(X)$ that will be of
interest for us. Of course, this is true also for $Q(X)$.

As we remarked in the previous section, the real valued bounded continuous functions on $\gamma(X)$ separate the
points of this space. It turns out that when $X$ is second countable a countable family of such functions will
suffice.

\begin{prop} \label{P:countable}

   If $X$ is second countable then there is a countable family of real valued bounded continuous functions on
   $\gamma(X)$ that separates the points of $\gamma(X)$.

\end{prop}

\begin{proof}

   From Proposition \ref{P:quot} we gather that it will be enough to show that such a countable family of functions
   exists on $Q(X)$. Recall that $(\Lp(X),\tau_s)$ is locally compact Hausdorff and second countable, in particular
   $\sigma$-compact. Let $\set{K_n}$ be an increasing sequence of $\tau_s$-compact subsets that covers $\Lp(X)$. Then
   $C(K_n)$, the algebra of all real continuous functions on the compact metrizable space $K_n$ is separable. Hence
   $\Cs(\Lp(X))\mid {}_{K_n}\subset C(K_n)$ is also separable. Thus there is a countable family $\set{f_n^m \mid
   1\leq m < \infty}\subset \Cs(\Lp(X))$ such that $\set{f_n^m\mid {}_{K_n} \mid 1\leq m < \infty}$ is dense in
   $\Cs(\Lp(X))\mid {}_{K_n}$. Now, if $g$ and $h$ are real bounded continuous functions on $Q(X)$ then

   \[
    sup\set{|g(y) - h(y)|\mid y\in Q(K_n)} = sup\set{|g\circ Q(S) - h\circ Q(S)|\mid S\in K_n}.
   \]

   Hence, viewing the elements of $\Cs(\Lp(X))$ as functions on $Q(X)$, the family $\set{f_n^m \mid 1\leq m <
   \infty}$ separates the points of $Q(K_n)$ and $\set{f_n^m \mid 1\leq m,n < \infty}$ separates the points of
   $Q(X)$.

\end{proof}

\begin{defn}

   A second countable locally compact space $X$ will be called a $\CR$-space if $\gamma(X)$ is a
   second countable locally compact Hausdorff space.

\end{defn}

In \cite{EW98} a separable $C^*$-algebra whose primitive ideal space is a $\CR$-space in the above terminology was
called a $\CR$-algebra. There the conditions imposed on the quotient of the primitive ideal space were for the
topology $\tau_{cr}$. Not every separable $C^*$-algebra is a $\CR$-algebra thus not every second countable locally
compact space is a $\CR$-space. An example is given in \cite[Example 9.2]{DH}. The class of $\CR$-algebras was found
in \cite{EW98} and \cite{EW01} useful for the study of certain $C^*$-dynamical systems and the corresponding crossed
products. It was remarked in \cite{EW01} that for a second countable locally compact space $X$ each of the following
properties is sufficient to ensure that it is a $\CR$-space: $X$ is Hausdorff, $X$ is compact, $\underset{H}\sim$ is
an open equivalence relation on $X$. Of course, this was done in \cite{EW01} only for the primitive ideal space of a
$C^*$-algebra so we shall reproduce and adapt the arguments for the general situation. The case of a Hausdorff space
is trivial. If $X$ is compact then $\gamma(X)$ is compact too. Recall that $\Lp(X)$ is $\ts$-compact. Theorem
\ref{T:iso} and the definition of the quotient topology on $\gamma(X)$ yield an isomorphism of the algebra
$C(\gamma(X))$ into the separable algebra $C(\Lp(X))$ hence $C(\gamma(X))$ and $\gamma(X)$ is second countable.
Alternatively, we can use the Proposition \ref{P:countable} to infer that $\gamma(X)$ is metrizable. Now suppose
that $X$ is a second countable locally compact space for which $\underset{H}\sim$ is an open equivalence relation.
Then by Proposition \ref{P:open}, $\gamma(X)$ is locally compact Hausdorff. Since the quotient map is continuous and
open it is easily seen that $\gamma(X)$ is second countable. We give below a characterization of the $\CR$-spaces.

\begin{thm} \label{T:CR}

   Let $X$ be a second countable locally compact space. The following conditions are equivalent:
   \begin{itemize}
      \item[(i)] $X$ is a $\CR$-space;
      \item[(ii)] $\gamma(X)$ is locally compact;
      \item[(iii)] $\gamma(X)$ is first countable.
   \end{itemize}

\end{thm}

\begin{proof}

   (i)$\Rightarrow$ (ii). This is immediate.

   (ii)$\Rightarrow$ (iii). By assumption $\gamma(X)$ is locally compact Hausdorff and $\sigma$-compact since $X$ is
   $\sigma$-compact. Then, by \cite[Theorem 7.2]{Du}, there is an increasing sequence of open sets $\set{U_n}$ in
   $\gamma(X)$ such that $\gamma(X) = \cup_{n=1} U_n$, $\overline{U}_n$ is compact and $\overline{U}_n\subset
   U_{n+1}$ for every $n$. Proposition \ref{P:countable} yields a sequence $\set{g_k}$ of continuous functions from
   $\gamma(X)$ to the interval $[0,1]$ that separates the points of $\gamma(X)$. The restrictions of the functions
   $\set{g_k}$ to the compact Hausdorff space $\overline{U}_n$ allow us to define a homeomorphism of
   $\overline{U}_n$ into $[0,1]^{\aleph_0}$. Hence $\overline{U}_n$ is metrizable and each open set $U_n$ is
   second countable. We conclude that $\gamma(X)$ is second countable. We actually proved that (i) follows from (ii)
   which is apparently more than we needed.

   (iii)$\Rightarrow$ (i). From Proposition \ref{P:quot} and the hypothesis it follows that $Q(X)$ is first countable. We
   know that it is a Hausdorff space. It has been noted above that $\Lp(X)$ is
   locally compact Hausdorff and under the present hypothesis on $X$ it is also second countable. Thus $Q(X)$ is a Hausdorff
   quotient of a second countable locally compact Hausdorff space. Then \cite[Theorem
   3]{S} implies that $Q(X)$ is locally compact and second countable and the same properties are shared by $\gamma(X)$ by
   Proposition \ref{P:quot}.

\end{proof}

Another characterization of $CR$-spaces can be given in terms of the quotient map $q$. A continuous map $\varphi$
from a topological space $Y$ onto a topological space $Z$ was called in \cite{M68} a bi-quotient map if for every
$z\in Z$ and every open cover of $\varphi^{-1}(z)$ there are finitely many sets $\set{U_i}$ in the cover such that
the interior of $\cup_{i=1} \varphi(U_i)$ is a neighbourhood of $z$. It follows from \cite[Proposition 3.3(d) and
Proposition 3.4]{M68} that whenever $\varphi$ is a quotient map of the second countable locally compact space $Y$
onto the Hausdorff space $Z$ then $Z$ is locally compact and second countable if and only if $\varphi$ is
bi-quotient. If we adapt this general result to our situation we get

\begin{prop}

   The second countable locally compact space $X$ is a $CR$-space if and only if the quotient map $q$ is
   bi-quotient.

\end{prop}
\bibliographystyle{amsplain}
\bibliography{}

\end{document}